\documentclass[12pt, reqno]{amsart}

\usepackage{amsmath,amsfonts,amssymb,amsthm,enumerate,hyperref,multicol}

\newtheorem{lemma}{Lemma}
\newtheorem{theorem}{Theorem}

\newtheorem{corollary}{Corollary}
\newtheorem{remark}{Remark}

\numberwithin{equation}{section}

\DeclareMathOperator*{\Sup}{\underset{x,t > 0}{Sup}}

\begin{document}

\leftline{ \scriptsize \it }

\title[Jain Operators]{On Sz\'{a}sz-Mirakyan-Jain Operators preserving exponential functions}
\maketitle

\begin{center}
{\bf G. C. Greubel} \\
Newport News, VA, United States \\
jthomae@gmail.com
\end{center}

\vspace{10mm}

\noindent \textbf{Abstract.} In the present article we define the Jain type modification of the generalized
Sz{\'a}sz-Mirakjan operators that preserve constant and exponential mappings. Moments, recurrence formulas, and other
identities are established for these operators. Approximation properties are also obtained with use of the
Boham-Korovkin theorem.

\vspace{2mm}
\noindent \textbf{Keywords.} Sz{\'a}sz-Mirakjan operators, Jain basis functions, Jain operators, 
Lambert W-function, Boham-Korovkin theorem. 

\smallskip
\noindent \textbf{2010 Mathematics Subject Classification}: 33E20, 41A25, 41A36.

\vspace{10mm}
\section{Introduction}
In Approximation theory positive linear operatos have been studied with the test functions $\{1, x, x^2 \}$ in order 
to determine the convergence of a function. Of interest are the Sz\'{a}sz-Mirakjan operators, based on the Poisson
distribution, which are useful in approximating functions on $[0, \infty)$ and are defined as, \cite{GMM}, \cite{OS},
\begin{align}\label{e1} 
S_{n}(f;x) = \sum_{k=0}^{\infty} \frac{(n x)^{k}}{k!} \, e^{- n x} \, f\left(\frac{k}{n}\right).
\end{align}

In 1972, Jain \cite{GCJ}, used the Lagrange expansion formula
\begin{align}\label{e2} 
\phi(z) = \phi(0) + \sum_{k=1}^{\infty} \frac{1}{k!} \, \left[ D^{k-1} \, \left( f^{k}(z) \, \phi'(z) \right) 
\right]_{z=0} \, \left(\frac{z}{f(z)}\right)^{k}
\end{align}
with $\phi(z) = e^{\alpha \, z}$ and $f(z) = e^{\beta \, z}$ to determined that
\begin{align}\label{e3} 
1 = \alpha \, \sum_{k=0}^{\infty} \frac{1}{k!} \, (\alpha + \beta k)^{k-1} \, z^{k} \, e^{- (
\alpha + \beta \, k) \, z}.
\end{align}
Jain established the basis functions
\begin{align}\label{e4} 
L_{n,k}^{(\beta)}(x) = \frac{n x \, (n x + \beta k)^{k-1}}{k!} \, e^{-(n x + \beta \, k)}
\end{align}
with the normalization 
\begin{align*}
\sum_{k=0}^{\infty} L_{n,k}^{(\beta)}(x) = 1
\end{align*}
and considered the operators
\begin{align}\label{e5} 
B_{n}^{\beta}(f,x) = \sum_{k=0}^{\infty} L_{n,k}^{(\beta)}(x) \, f\left(\frac{k}{n}\right) \hspace{15mm} 
x \in [0,\infty).
\end{align}
In the reduction of $\beta = 0$ the Jain operators reduce to the Sz\'{a}sz-Mirakjan operators.

Recently Acar, Aral, and Gonska \cite{AAG} considered the Sz\'{a}sz-Mirakjan operators which preserve the 
test functions $\{1, e^{a x} \}$ and established the operators
\begin{align}\label{e6} 
R_{n}^{*}(f;x) = e^{-n \, \gamma_{n}(x)} \, \sum_{k=0}^{\infty} \frac{(n \, \gamma_{n}(x))^{k}}{k!} \, 
f\left(\frac{k}{n}\right)
\end{align}
for functions $f \in C[0,\infty)$, $x \geq 0$, and $n \in \mathbb{N}$ with the reservation property
\begin{align}\label{e7} 
R_{n}^{*}(e^{2 a t};x) = e^{2 a x}.
\end{align}
Here the Jain basis is used to extend the the class of operators for the test functions $\{1, e^{- \lambda x}\}$ by 
defining Sz\'{a}sz-Mirakyan-Jain operators which preserve the mapping of $e^{- \lambda x}$, for $\lambda, x > 0$. 
In the case of $\lambda=0$ the Sz\'{a}sz-Mirakyan-Jain operators are constant preserving operators. Moments, recurrence
formulas, and other identities are established for these new operators. Approximation properties are also
obtained with use of the Boham-Korovkin theorem. The Lambert W-function and related properties are used 
in the analaysis of the properties obtained for the Sz\'{a}sz-Mirakyan-Jain operators.

\section{Sz\'{a}sz-Mirakyan-Jain Operators}
The Sz\'{a}sz-Mirakyan-Jain operators, (SMJ), which are a generalization of the Sz\'{a}sz-Mirakyan operators, are defined by
\begin{align}\label{e8} 
R_{n}^{(\beta)}(f;x) = n \, \alpha_{n}(x) \, \sum_{k=0}^{\infty} \frac{1}{k!} \, (n \, \alpha_{n}(x) + \beta \, k)^{k-1}
\, e^{-(n \, \alpha_{n}(x) + \beta \, k)} \, f\left(\frac{k}{n}\right)
\end{align}
for $f \in C[0,\infty)$. It is required that these operators preserve the mapping of $e^{- \lambda \, x}$, as 
given by
\begin{align}\label{e9} 
R_{n}^{(\beta)}(e^{ - \lambda t} ; x) = e^{- \lambda x}
\end{align}
where $x \geq 0$ and $n \in \mathbb{N}$, and $\lambda \geq 0$. When $\beta = 0$ in \eqref{e8} the operator reduces to that defined 
by Acar, Aral, and Gonska \cite{AAG}. When $\beta = 0$ and $\alpha_{n}(x) = x$ the operator reduces to the well
known Sz\'{a}sz-Mirakyan operators given by \eqref{e1}. For $0 \leq \beta < 1$ and $\alpha_{n}(x) = x$ these operators
reduce to the Sz\'{a}sz-Mirakyan-Durrmeyer operators defined by Gupta and Greubel in \cite{GG1}.

\begin{lemma} \label{L1} 
For $x \geq 0, \lambda \geq 0$, we have
\begin{align}\label{e10} 
\alpha_{n}(x) = \frac{ - \lambda \, x}{n \, (z(\lambda/n, \beta) - 1)},
\end{align}
where $-\beta \, z(t, \beta) = W(-\beta \, e^{-\beta - t})$ and $W(x)$ is the Lambert W-function.
\end{lemma}

\begin{proof}
Considering the mapping \eqref{e9} it is required that
\begin{align}\label{e11} 
e^{- \lambda x} &= n \, \alpha_{n}(x) \, \sum_{k=0}^{\infty} \frac{1}{k!} \, (n \, \alpha_{n}(x) + \beta \, k)^{k-1}
\, e^{-(n \, \alpha_{n}(x) + \beta k)} \,  e^{- \lambda k/n}
\end{align}
Making use of \eqref{e3} in the form
\begin{align}\label{e12} 
e^{n \, \alpha_{n}(x) \, z} = n \, \alpha_{n}(x) \, \sum_{k=0}^{\infty} \frac{1}{k!} \, (n \, \alpha_{n}(x) + 
\beta k )^{k-1} \, e^{-(\beta z - \ln(z)) k}
\end{align}
and letting $\beta \, z - \ln(z) = \beta + \frac{\lambda}{n}$ then 
\begin{align*}
e^{n \, \alpha_{n}(x) \, z} = n \, \alpha_{n}(x) \, \sum_{k=0}^{\infty} \frac{1}{k!} \, (n \, \alpha_{n}(x) + 
\beta k)^{k-1} \, e^{-(\beta + \lambda/n) \, k}
\end{align*}
which provides
\begin{align*}
e^{-\lambda x} = e^{n \, \alpha_{n}(x) \, (z-1)}
\end{align*}
or 
\begin{align*}
\alpha_{n}(x) = - \frac{\lambda \, x}{n \, (z(\lambda/n, \beta) - 1)}.
\end{align*}
The value of $z$ is determined by the equation $\beta \, z - \ln(z) = \beta + \frac{\lambda}{n}$ which can be seen 
in the form
\begin{align*}
z \, e^{- \beta \, z} = e^{- \beta - \lambda/n}
\end{align*}
and has the solution
\begin{align}\label{e13} 
z(\lambda/n, \beta) = - \frac{1}{\beta} \, W(- \beta \, e^{- \beta - \lambda/n}),
\end{align}
where $W(x)$ is the Lambert W-function.
\end{proof}

\begin{remark}
For the case $\lambda \to 0$ the resulting $\alpha_{n}(x)$ is 
\begin{align*}
\lim_{\lambda \to 0} \, \alpha_{n}(x) = (1-\beta) \, x. 
\end{align*}
\end{remark}


\begin{proof}
For the case $\lambda \to 0$ the resulting $z = z(\lambda/n, \beta)$ of \eqref{e13} yields $z(0, \beta) = 1$.
By considering 
\begin{align*}
\frac{\partial z}{\partial \lambda} = - \frac{1}{\beta} \, \frac{\partial}{\partial \lambda} \, W(- \beta \, e^{- \beta - \lambda/n}) 
= \frac{W(- \beta \, e^{- \beta - \lambda/n})}{n \, \beta \, ( 1 + W(- \beta \, e^{- \beta - \lambda/n}))}
\end{align*}
and 
\begin{align*}
\lim_{\lambda \to 0} \, \frac{\partial z}{\partial \lambda} = - \frac{1}{n \, ( 1 - \beta)}.
\end{align*}
Now, by use of L'Hospital's rule, 
\begin{align*}
\lim_{\lambda \to 0} \, \alpha_{n}(x) &= \frac{x}{n} \, \lim_{\lambda \to 0} \frac{\lambda}{z - 1} = \frac{x}{n}
\, \lim_{\lambda \to 0} \frac{1}{\frac{\partial z}{\partial \lambda}} = (1-\beta) \, x
\end{align*}
as claimed.
\end{proof}


By taking the case of $\lambda \to 0$ the operators $R_{n}^{(\beta)}(f;x)$ reduce from exponential preserving to 
constant preserving operators. In this case the operators $R_{n}^{(\beta)}(f;x)|_{\lambda \to 0}$ are related to the
Jain operators, \eqref{e5}, by $R_{n}^{(\beta)}(f; x) = B_{n}^{\beta}(f;(1- \beta) \, x)$.

The SMJ operators are now completely defined by
\begin{equation}\label{e14} 
\left \{
\begin{aligned}
R_{n}^{(\beta)}(f;x) &= n \, \alpha_{n}(x) \, \sum_{k=0}^{\infty} \frac{1}{k!} \, (n \, \alpha_{n}(x) + \beta \, k)^{k-1}
\, e^{-(n \, \alpha_{n}(x) + \beta \, k)} \, f\left(\frac{k}{n}\right), \\
\alpha_{n}(x) &= - \frac{\lambda \, x}{n \, (z(\lambda/n, \beta) - 1)}, \\
z(t, \beta) &= - \frac{1}{\beta} \, W(- \beta \, e^{- \beta - t})
\end{aligned} \right.
\end{equation} 
and the requirement that $R_{n}^{(\beta)}(e^{- \lambda t} ; x) = e^{-\lambda x}$, for $x \geq 0$, $\lambda \geq 0$ and
$n \in \mathbb{N}$.

\section{Moment Estimations}

\begin{lemma} \label{L2} 
The moments for the SMJ operators are given by:
\begin{align}\label{e15} 
R_{n}^{(\beta)}(1; x) &= 1 \nonumber\\
R_{n}^{(\beta)}(t; x) &= \frac{\alpha_{n}(x)}{1-\beta} \nonumber\\
R_{n}^{(\beta)}(t^{2}; x) &= \frac{\alpha_{n}^{2}(x)}{(1-\beta)^{2}} + \frac{\alpha_{n}(x)}{n \, (1-\beta)^{3}} \\
R_{n}^{(\beta)}(t^{3}; x) &= \frac{\alpha_{n}^{3}(x)}{(1-\beta)^{3}} + \frac{3 \, \alpha_{n}^{2}(x)}{n \, (1-\beta)^{4}} +
(1 + 2 \, \beta) \, \frac{\alpha_{n}(x)}{n^{2} \, (1-\beta)^{5}} \nonumber\\
R_{n}^{(\beta)}(t^{4}; x) &= \frac{\alpha_{n}^{4}(x)}{(1-\beta)^{4}} + \frac{6 \, \alpha_{n}^{3}(x)}{n \, (1-\beta)^{5}} +
(7 + 8 \, \beta) \, \frac{\alpha_{n}^{2}(x)}{n^{2} \, (1-\beta)^{6}} + (1 + 8 \beta + 6 \beta^{2}) \, 
\frac{\alpha_{n}(x)}{n^{3} \, (1-\beta)^{7}} \nonumber\\
R_{n}^{(\beta)}(t^{5}; x) &= \frac{\alpha_{n}^{5}(x)}{(1-\beta)^{5}} + \frac{10 \, \alpha_{n}^{4}(x)}{n \, (1-\beta)^{6}} +
5 \, (5 + 4 \, \beta) \, \frac{\alpha_{n}^{3}(x)}{n^{2} \, (1-\beta)^{7}} \nonumber\\
& \hspace{10mm} + 15 \, (1 + 4 \beta + 2 \beta^{2}) \, \frac{\alpha_{n}^{2}(x)}{n^{3} \, (1-\beta)^{8}} 
+ (1 + 22 \beta + 58 \beta^{2} + 24 \beta^{3}) \, \frac{\alpha_{n}(x)}{n^{4} \, (1-\beta)^{9}}. 
\nonumber
\end{align}
\end{lemma}
The proof follows directly from work of the author dealing with moment operators for the Jain basis, 
see \cite{G1, GG1, GG2}.

\begin{lemma} \label{L3} 
Let, $\phi = t - x$, then the central moments of the SMJ operators are:
\begin{align}\label{e16} 
R_{n}^{(\beta)}(\phi^{0}; x) &= 1 \nonumber\\
R_{n}^{(\beta)}(\phi^{1}; x) &= \frac{\alpha_{n}(x)}{1-\beta} - x \nonumber\\
R_{n}^{(\beta)}(\phi^{2}; x) &= \left(\frac{\alpha_{n}(x)}{1-\beta} - x \right)^{2} + \frac{\alpha_{n}(x)}{n \, (1-\beta)^{3}} \\
R_{n}^{(\beta)}(\phi^{3}; x) &= \left(\frac{\alpha_{n}(x)}{1-\beta} - x \right)^{3} + \frac{3 \, \alpha_{n}(x)}{n \, (1-\beta)^{3}}
\, \left(\frac{\alpha_{n}(x)}{1-\beta} - x \right) + (1 + 2 \, \beta) \, \frac{\alpha_{n}(x)}{n^{2} \, (1-\beta)^{5}} 
\nonumber
\end{align}
\begin{align*}
R_{n}^{(\beta)}(\phi^{4}; x) &= \left(\frac{\alpha_{n}(x)}{1-\beta} -x \right)^{4} + \frac{6 \, \alpha_{n}(x)}{n \, (1-\beta)^{3}}
\, \left(\frac{\alpha_{n}(x)}{1-\beta} - x \right)^{2} +
(7 + 8 \, \beta) \, \frac{\alpha_{n}(x)}{n^{2} \, (1-\beta)^{5}} \nonumber\\
& \hspace{10mm} \cdot \left(\frac{\alpha_{n}(x)}{1-\beta} - x \right) + (1 + 8 \beta + 6 \beta^{2}) \, 
\frac{\alpha_{n}(x)}{n^{3} \, (1-\beta)^{7}} + \frac{3 \, \alpha_{n}(x)}{n^{2} \, (1-\beta)^{5}} \nonumber\\
R_{n}^{(\beta)}(\phi^{5}; x) &= \left(\frac{\alpha_{n}(x)}{1-\beta} - x\right)^{5} + \frac{10 \, \alpha_{n}(x)}{n \, (1-\beta)^{3}} 
\left( \frac{\alpha_{n}(x)}{1-\beta} - x \right)^{3} + \frac{5 \,\alpha_{n}(x)}{n^{2} \, (1-\beta)^{5}} \left(
\frac{\alpha_{n}(x)}{1-\beta} - x\right) \cdot \mu_{1}  \nonumber\\
& \hspace{15mm} + \frac{5 \, \alpha_{n}(x)}{n^{3} \, (1-\beta)^{7}} \cdot \mu_{2} + (1 + 22 \beta + 58 \beta^{2}
+ 24 \beta^{3}) \, \frac{\alpha_{n}(x)}{n^{4} \, (1-\beta)^{9}}, 
\nonumber
\end{align*}
\end{lemma}
where 
\begin{align*}
\mu_{1} &= (5 + 4 \beta) \, \left( \frac{\alpha_{n}(x)}{1 - \beta} - x \right) + 3 \, x \\
\mu_{2} &= 3 \, (1 + 4 \beta + 2 \beta^{2}) \, \left( \frac{\alpha_{n}(x)}{1-\beta} - x \right) + 2 \, (1 + 2
\beta) \, x.
\end{align*}

\begin{proof}
Utilizing the binomial expansion
\begin{align*}
\phi^{m} = (t - x)^{m} = \sum_{k=0}^{m} (-1)^{k} \, \binom{m}{k} \, t^{m-k} \, x^{k}
\end{align*}
then
\begin{align}\label{e17} 
R_{n}^{(\beta)}(\phi^{m}; x) = \sum_{k=0}^{m} (-1)^{k} \, \binom{m}{k} \, x^{k} \, R_{n}^{(\beta)}(t^{m-k}; x).
\end{align}
Wih the use of \eqref{e15} the first few values of $m$ are:
\begin{align*}
R_{n}^{(\beta)}(\phi^{0}; x) &= R_{n}^{(\beta)}(t^{0}; x) = 1 \\
R_{n}^{(\beta)}(\phi^{1}; x) &= R_{n}^{(\beta)}(t; x) - x \, R_{n}^{(\beta)}(t^{0}; x) = \frac{\alpha_{n}(x)}{1-\beta} - x \\
R_{n}^{(\beta)}(\phi^{2}; x) &= R_{n}^{(\beta)}(t^{2}; x) - 2 x \, R_{n}^{(\beta)}(t;x) + x^{2} \, R_{n}^{(\beta)}(t^{0}; x) \\
&= \frac{\alpha_{n}^{2}(x)}{(1-\beta)^{2}} + \frac{\alpha_{n}(x)}{n \, (1-\beta)^{3}} - 2x \, \frac{\alpha_{n}(x)}
{1-\beta} + x^{2} \\
&= \left( \frac{\alpha_{n}(x)}{1 - \beta} - x\right)^{2} + \frac{\alpha_{n}(x)}{n \, (1-\beta)^{3}}
\end{align*}
The remainder of the central moments follow from \eqref{e15} and \eqref{e17}.
\end{proof}

\begin{lemma} \label{L4}
The central moments, given in Lemma 3, lead to the limits: 
\begin{equation}
\begin{aligned}\label{e18} 
\lim_{n \to \infty} \, n \, R_{n}^{(\beta)}(\phi; x) &= \frac{\lambda \, x}{2! \, (1-\beta)^{2}} \\
\lim_{n \to \infty} \, n \, R_{n}^{(\beta)}(\phi^{2}; x) &= \frac{x}{(1-\beta)^{2}}
\end{aligned}
\end{equation}

\end{lemma}

\begin{proof}
By setting $t = \lambda/n$ in \eqref{a4} then
\begin{align*}
& \frac{(- \lambda)}{n \, (1-\beta) \, (z(\lambda/n,\beta) - 1)} = 1 + \frac{v}{2!} + 2\, (1-4\beta) \, \frac{v^{2}}{4!} + 
6 \beta^{2} \, \frac{v^{3}}{4!} \nonumber\\ 
& \hspace{10mm} - (1 - 8\beta + 88\beta^{2} + 144\beta^{3}) \, \frac{v^{4}}{6!}  + 840 \beta^{2}\,(1 + 12\beta + 
8\beta^{2}) \, \frac{v^{5}}{8!} + O(v^{6}),
\end{align*}
where $n \, (1-\beta)^{2} \, v = \lambda$. This expansion may be placed into the form
\begin{align*}
\frac{\alpha_{n}(x)}{1-\beta} - x &= \frac{v \, x}{2!} \, \left( 1 + (1-4\beta) \, \frac{v}{3!} + 12\beta^{2} \, 
\frac{v^{2}}{4!} - O(v^{3}) \right).
\end{align*}
Multiplying by $n$ and taking the desired limit the resulting value is given by
\begin{align*}
\lim_{n \to \infty} \, n \, R_{n}^{(\beta)}(\phi; x) = \frac{\lambda \, x}{2! \, (1-\beta)^{2}}.
\end{align*}
It is evident that
\begin{align*}
\left(\frac{\alpha_{n}(x)}{1-\beta} - x \right)^{2} &= \left(\frac{v \, x}{2!}\right)^{2} \, \left( 1 + 2(1-4\beta)
\, \frac{v}{3!} + 20 \, (1-8\beta + 52\beta^{2}) \, \frac{v^{2}}{6!} - O(v^{3}) \right)
\end{align*}
for which
\begin{align*}
\left( \frac{\alpha_{n}(x)}{1-\beta} - x\right)^{2} &+ \frac{\alpha_{n}(x)}{n \, (1-\beta)^{3}} \\
&= \left(\frac{v \, x}{2!}\right)^{2} \, \left( 1 + 2(1-4\beta) \, \frac{v}{3!} + 20 \, (1-8\beta + 52\beta^{2}) \, 
\frac{v^{2}}{6!} - O(v^{3}) \right) \\
& \hspace{5mm} + \frac{x}{n \, (1-\beta)^{2}} \, \left(1 + \frac{v}{2!} + 2\, (1-4\beta) \, \frac{v^{2}}{4!} - O(v^{3}) \right)
\end{align*}
Multiplying by $n$ and taking the limit yields
\begin{align*}
\lim_{n \to \infty} \, n \, R_{n}^{(\beta)}(\phi^{2}; x) = \frac{x}{(1-\beta)^{2}}.
\end{align*}
\end{proof}


\begin{remark}
Other limits may be determined by extending the work of Lemma 4, such as:
\begin{equation}\label{e19} 
\begin{aligned}
\lim_{n\to \infty} R_{n}^{(\beta)}(\phi^{m}; x) &= 0, \mbox{ for } m \geq 1 \\
\lim_{n \to \infty} n \, R_{n}^{(\beta)}(\phi^{m}; x) &= 0, \mbox{ for } m \geq 3 \\
\lim_{n \to \infty} n^{2} \, R_{n}^{(\beta)}(\phi^{3}; x) &= \frac{2(1 + 2 \beta) \, x + 3 \lambda \, x^{2}}{2! \, 
(1-\beta)^{4}} \\
\lim_{n \to \infty} n^{2} \, R_{n}^{(\beta)}(\phi^{4}; x) &= \frac{3 \, x^{2}}{(1-\beta)^{4}} 
\end{aligned}
\end{equation}
\end{remark}



\begin{lemma} \label{L5}
Expansion on a general exponential weight is given by
\begin{align*}
R_{n}^{(\beta)}(e^{-\mu \, t} ; x) = e^{n \, \alpha_{n}(x) \, (z(\mu/n, \beta) - 1)},
\end{align*}
or
\begin{align}\label{e20} 
R_{n}^{(\beta)}(e^{- \mu \, t} ; x) = Exp\left[ -\lambda \, x \, \left(\frac{z(\mu/n,\beta) - 1}{z(\lambda/n, \beta)-1} \right)
\right] = Exp\left[ - \mu \, x \cdot \frac{\lambda}{\mu} \, \frac{z(\mu/n, \beta) - 1}{z(\lambda/n, \beta) - 1} \right]
\end{align} 
for $\mu \geq 0$ and has the expansion 
\begin{equation}\label{e21} 
\begin{aligned}
R_{n}^{(\beta)}(e^{-\mu \, t}; x) &= e^{-\mu \, x} \, \left( 1 + \frac{\mu (\mu - \lambda) x}{2! \, n (1-\beta)^{2}} + 
( (3 \mu x - 4 - 8 \beta) \mu \right. \\
& \hspace{10mm} \left. - (3 \mu x - 2 + 8 \beta) \lambda) \, \frac{\mu (\mu - \lambda) x}{4! \, n^{2} (1-\beta)^{4}} 
+ \mathcal{O}\left(\frac{\mu (\mu - \lambda) x}{6! \, n^{3} (1-\beta)^{6}}\right) \right)
\end{aligned}
\end{equation}
where $-\beta \, z(\mu/n,\beta) = W(- \beta \, e^{-\beta - \mu/n})$, $-\beta \, z(\lambda/n, \beta) = W(- \beta \, 
e^{- \beta - \lambda/n})$. In the limit as $n \to \infty$ it is evident that
\begin{align}\label{e22} 
\begin{split}
\lim_{n \to \infty} R_{n}^{(\beta)}(e^{-\mu t}; x) &= e^{-\mu x} \\
\lim_{n \to \infty} n \, \left[ R_{n}^{(\beta)}(e^{- \mu t}; x) - e^{- \mu x} \right] &= \frac{\mu (\mu - \lambda) \, x}{ 
2! \, (1-\beta)^{2}} \, e^{- \mu x}.
\end{split}
\end{align}
\end{lemma}

\begin{proof}
It is fairly evident that
\begin{align*}
R_{n}^{(\beta)}(e^{-\mu t}; x) &= n \alpha_{n}(x) \, \sum_{k=0}^{\infty} \frac{1}{k!} \, (n \, \alpha_{n}(x) + \beta k)^{k-1} \,
e^{-n \, \alpha_{n}(x) - (\beta + \mu/n) k}
\end{align*}
which, by comparison to \eqref{e12}, leads to
\begin{align*}
R_{n}^{(\beta)}(e^{-\mu t}; x) = e^{- n \, \alpha_{n}(x) \, (z(\mu/n, \beta) - 1)} = Exp\left[- \lambda \, x \, 
\left(\frac{z(\mu/n,\beta)-1}{z(\lambda/n,\beta)-1}\right) \right].
\end{align*}
The expansion of \eqref{e20}, with use of \eqref{a5}, is given by
\begin{align*}
R_{n}^{(\beta)}(e^{- \mu t}; x) &= \sum_{k=0}^{\infty} \frac{(-\mu x)^{k}}{k!} \, \left(\frac{\lambda}{\mu} \, \frac{z(\mu/n, 
\beta) - 1}{z(\lambda/n, \beta) - 1} \right)^{k} \\
&= \sum_{k=0}^{\infty} \frac{(-\mu x)^{k}}{k!} \, \left(1 - \frac{k (\mu - \lambda)}{2! \, (1-\beta)^{2}} + k ( (3k + 1 + 
8 \beta) \mu \right. \\
& \hspace{5mm} \left. + (3k -1 - 8 \beta) \lambda ) \, \frac{\mu - \lambda}{4! \, (1-\beta)^{4}} + \mathcal{O}\left( 
\frac{\mu - \lambda}{6! \, (1-\beta)^{6}} \right) \right) \\
&= e^{- \mu x} \, \left( 1 + \frac{\mu (\mu - \lambda) x}{2! \, n (1-\beta)^{2}} + ( (3 \mu x - 4 - 8 \beta) \mu \right. \\
& \hspace{10mm} \left. - (3 \mu x - 2 + 8 \beta) \lambda) \, \frac{\mu (\mu - \lambda) x}{4! \, n^{2} (1-\beta)^{4}} 
+ \mathcal{O}\left(\frac{\mu (\mu - \lambda) x}{6! \, n^{3} (1-\beta)^{6}}\right) \right).
\end{align*}
Taking the appropriate limits yields the desired results.
\end{proof}

\begin{remark}
By use of Lemma 5 it may be stated that:
\begin{align}\label{e23} 
\lim_{n\to \infty} n^{2} \, R_{n}^{(\beta)}((e^{-t} - e^{-x})^{4}; x) = \frac{3 \, x^{2} \, e^{-4 x}}{(1 - \beta)^{4}}. 
\end{align}
\end{remark}


\begin{proof}
Since
\begin{align*}
R_{n}^{(\beta)}((e^{-t} - e^{-x})^{4};x) &= R_{n}^{(\beta)}(e^{-4 t}; x) - 4 \, e^{-x} \, R_{n}^{(\beta)}(e^{-3 t}; x) + 6 \, 
e^{-2 x} \, R_{n}^{(\beta)}(e^{-2 t}; x) \nonumber\\
& \hspace{10mm} - 4 \, e^{-3 x} \, R_{n}^{(\beta)}(e^{-t}; x) + e^{-4 x} \, R_{n}^{(\beta)}(1; x)
\end{align*}
then, by making use of \eqref{e21}, it becomes evident that
\begin{align*}
R_{n}^{(\beta)}((e^{-t} - e^{-x})^{4}; x) = \frac{3 \, x^{2} \, e^{-4 x}}{n^{2} \, (1-\beta)^{4}} + \mathcal{O}\left( 
\frac{1}{n^{3} \, (1-\beta)^{6}}
\right).
\end{align*}
Multiplying by $n^{2}$ and taking the limit as $n \to \infty$ yields the desired result.
\end{proof}


\vspace{5mm}
\section{Analysis}

\begin{theorem}
Given the sequence $A_{n} : C^{*}[0,\infty) \to C^{*}[0,\infty)$of positive linear operators which satisfies 
the conditions
\begin{align*}
\lim_{n \to \infty} A_{n}(e^{-k t}; x) = e^{-k x}, \hspace{5mm} k=0,1,2
\end{align*}
uniformly in $[0,\infty)$ then
\begin{align*}
\lim_{n \to \infty} A_{n}(f; x) = f(x)
\end{align*}
uniformly in $[0,\infty)$ for every $f \in C^{*}[0,\infty)$.
\end{theorem}


The proof of this theorem 1 can be found in \cite{Alt-Camp,Boy-Ves,Holhos} and has, in essense, been demonstarted by \eqref{e21} for 
$\mu \geq 0$. An estimate of the rate of convergence for the SMJ operators will require the use of the modulus of continuity
\begin{align*}
\omega(F,\delta) = \Sup \, |F(t) - F(x)|
\end{align*}
and can be seen as, for the case where $F(e^{-t}) = f(t)$,
\begin{align*}
\omega^{*}(f;\delta) = \Sup_{|e^{-t} - e^{-x}| \leq \delta} \, |f(t) - f(x)|
\end{align*}
and is well defined for $\delta \geq 0$ and all functions $f \in C^{*}[0,\infty)$. In the present case the modulus of continuity
has the property
\begin{align}\label{e24} 
|f(t) - f(x)| \leq \left(1 + \frac{(e^{-x} - e^{-t})^{2}}{\delta^{2}}\right) \, \omega^{*}(f;\delta), \hspace{5mm} \delta > 0. 
\end{align}
Further properties and use of the modulus of continuity can be found in \cite{Boy-Ves, Holhos}. The following theorem can also 
be found in the later.

\begin{theorem}
If a sequence of positive linear operators $A_{n} : C^{*}[0,\infty) \to C^{*}[0,\infty)$ satisfy the equalities:
\begin{align*}
\|A_{n}(1 ;x) - 1 \|_{[0,\infty)} &= a_{n} \\
\|A_{n}(e^{-t};x) - e^{-x} \|_{[0,\infty)} &= b_{n} \\
\|A_{n}(e^{-2 t};x) - e^{-2 x} \|_{[0,\infty)} &= c_{n},
\end{align*}
where $a_{n}, b_{n}$ and $c_{n}$ are bounded and finite, in the limit $n \to \infty$, then
\begin{align*}
\| A_{n}(f; x) - f(x) \|_{[0, \infty)} \leq a_{n} \, |f(x)| + (2 + a_{n}) \, \omega^{*}(f, \sqrt{a_{n} + 2 \, b_{n} + c_{n}}),
\end{align*}
for every function $f \in C^{*}[0,\infty)$, and satisfies
\begin{align*}
\| A_{n}(f; x) - f(x) \|_{[0, \infty)} \leq 2 \, \omega^{*}(f, \sqrt{2 \, b_{n} + c_{n}})
\end{align*}
for constant preserving operators.
\end{theorem}


\begin{proof}
Since
\begin{align*}
A_{n}((e^{-t} - e^{-x})^{2}; x) = [A_{n}(e^{-2 t}; x) - e^{-2 x}] - 2 \, e^{-x} \, [A_{n}(e^{-t}; x) - e^{-x}] + e^{-2x} 
\, [A_{n}(1;x) - 1]
\end{align*}
then, by use of \eqref{e24}, 
\begin{align*}
A_{n}(|f(t) - f(x)|; x) &\leq \left( A_{n}(1;x) + \frac{1}{\delta^{2}} \, A_{n}((e^{-t} - e^{-x})^{2}; x) \right) \, 
\omega^{*}(f,\delta) \\
&\leq \left( 1 + a_{n} + \frac{a_{n} + 2 \, b_{n} + c_{n}}{\delta^{2}} \right) \, \omega^{*}(f,\delta).
\end{align*}
By choosing $\delta = \sqrt{a_{n} + 2 \, b_{n} + c_{n}}$ then
\begin{align*}
A_{n}(|f(t) - f(x)|; x) &\leq ( 2 + a_{n}) \, \omega^{*}(f,\sqrt{a_{n} + 2 \, b_{n} + c_{n}}).
\end{align*}
Now, making use of
\begin{align*}
|A_{n}(f;x) - f(x)| \leq |f| \, |A_{n}(1;x) - 1| + A_{n}(|f(t) - f(x)|; x) 
\end{align*}
leads to the uniform estimation of convergence in the form
\begin{align*}
\| A_{n}(f; x) - f(x) \|_{[0, \infty)} \leq a_{n} \, |f(x)| + (2 + a_{n}) \, \omega^{*}(f, \sqrt{a_{n} + 2 \, b_{n} + c_{n}}).
\end{align*}
For constant preserving operators the property $\|A_{n}(1;x) - 1\|_{[0,\infty)} = a_{n} = 0$ holds and leads to
\begin{align*}
\| A_{n}(f; x) - f(x) \|_{[0, \infty)} \leq 2 \, \omega^{*}(f, \sqrt{2 \, b_{n} + c_{n}}).
\end{align*}
\end{proof}


\begin{remark}
The SMJ operators satisfy
\begin{align*}
\| R_{n}^{(\beta)}(f; x) - f(x) \|_{[0, \infty)} \leq 2 \, \omega^{*}(f, \sqrt{2 \, b_{n} + c_{n}}).
\end{align*}
\end{remark}


\begin{proof}
By using Lemma 2 it is evident that $R_{n}^{(\beta)}(1; x) = 1$ and yields $a_{n} = 0$. By using \eqref{e21}, of 
Lemma 5, it is seen that
\begin{align*}
R_{n}^{(\beta)}(e^{-\mu \, t}; x) - e^{-\mu x} &= e^{-\mu \, x} \, \left( \frac{\mu (\mu - \lambda)x}{2! \, n \, (1-\beta)^{2}} 
- \frac{\Lambda(x, \mu, \lambda) \, \mu (\mu - \lambda) x}{4! \, n^{2} \, (1-\beta)^{4}} + \mathcal{O}\left( \frac{1}{n^{3} \, 
(1-\beta)^{6}}
\right) \right),
\end{align*}
where $\Lambda(x, \mu, \lambda) = (3 \mu x - 4 - 8 \beta) \, \mu - (3 \mu x - 2 + 8 \beta) \lambda$, and provides
\begin{align*}
\| R_{n}^{(\beta)}(e^{-\mu \, t}; x) - e^{-\mu x} \| = \left\| \frac{\mu (\mu - \lambda) \, x \, e^{-\mu \, x}}{
2! \, n (1-\beta)^{2}} \, \left( 1 + \frac{2 \, \Lambda(x, \mu, \lambda)}{4! \, n \, (1-\beta)^{2}} + \mathcal{O}\left( 
\frac{1}{n^{2} \, (1-\beta)^{4}} \right) \right) \right\|
\end{align*}
which, for $\mu \in \{1,2\}$, the remaining limiting values, $b_{n}$ and $c_{n}$ can be seen to be bounded and 
finite. It is also evident that in the limiting case, $n \to \infty$, $b_{n}$ and $c_{n}$ tend to zero. By the 
resulting statements of Theorem 2 it is determined that
\begin{align*}
\| R_{n}^{(\beta)}(f; x) - f(x) \|_{[0, \infty)} \leq 2 \, \omega^{*}(f, \sqrt{2 \, b_{n} + c_{n}}).
\end{align*}
as claimed.
\end{proof}



For the SMJ operators a quantitative Voronovskaya-type theorem can be defined in the following way.
\begin{theorem}
Let $f, f', f'' \in C^{*}[0,\infty)$ then
\begin{align*}
&\left| n \, [R_{n}^{(\beta)}(f;x) - f(x)] - \frac{\lambda \, x}{2! \, (1-\beta)^{2}} \, f'(x) - \frac{x}{n \, (1-\beta)^{2}}
\, f''(x) \right| \\
& \hspace{5mm} \leq |\mu_{n}(x,\beta)| \, |f'(x)|+ |\nu_{n}(x,\beta)| \, |f''(x)| \\
& \hspace{15mm} + 2 \, (2 \, \nu_{n}(x,\beta) + \frac{x}
{(1-\beta)^{2}} + \zeta_{n}(x,\beta) ) \, \omega^{*}\left(f''; \frac{1}{\sqrt{n}}\right)
\end{align*}
where
\begin{align*}
\mu_{n}(x, \beta) &= n \, R_{n}^{(\beta)}(\phi; x) - \frac{\lambda \, x}{2! \, (1-\beta)^{2}} \\
\nu_{n}(x, \beta) &= \frac{1}{2!} \, \left(n \, R_{n}^{(\beta)}(\phi^{2}; x) - \frac{x}{(1-\beta)^{2}} \right) \\
\zeta_{n}(x, \beta) &= n^{2} \, \sqrt{R_{n}^{(\beta)}((e^{-x} - e^{-t})^{4}; x)} \, \sqrt{R_{n}^{(\beta)}(\phi^{4}; x)}.
\end{align*}
\end{theorem}

\begin{proof}
The Taylor expansion for the function $f(x)$ is seen by
\begin{align}\label{e25} 
f(t) = f(x) + f'(x) \, (t-x) + \frac{f''(x)}{2!} \, (t-x)^{2} + \theta(t,x) \, (t-x)^{2} 
\end{align}
where $2 \, \theta(t,x) = f''(\eta) - f''(x)$ for $x \leq \eta \leq t$. Applying the SMJ operator to the Taylor
expansion it is determined that
\begin{align*}
& |R_{n}^{(\beta)}(f(t);x) - f(x) \, R_{n}^{(\beta)}(1;x) - f'(x) \, R_{n}^{(\beta)}(\phi;x) - \frac{f''(x)}{2!} 
\, R_{n}^{(\beta)}(\phi^{2};x)|  \\
& \hspace{30mm}\leq |R_{n}^{(\beta)}(\theta(t,x) \, \phi^{2};x)|.
\end{align*}
Using the results of lemma 4 and 5 this can be seen by
\begin{align*}
& \left|n \, \left( R_{n}^{(\beta)}(f;x) - f(x) \right) - \frac{\lambda \, x}{2! \, (1-\beta)^{2}} \, f'(x) - \frac{x}{2! \, 
(1-\beta)^{2}} \, f''(x) \right|  \\
& \hspace{10mm} \leq \left| n \, R_{n}^{(\beta)}(\phi;x) - \frac{\lambda \, x}{2! \, (1-\beta)^{2}} \right| \, |f'(x)|+ \frac{1}{2!} 
\, \left| n \, R_{n}^{(\beta)}(\phi^{2};x) - \frac{x}{(1-\beta)^{2}} \right| \, |f''(x)| \\
& \hspace{20mm} + |n \, R_{n}^{(\beta)}(\theta(t,x) \, \phi^{2}; x)|
\end{align*}
or
\begin{align*}
& \left|n \, \left( R_{n}^{(\beta)}(f;x) - f(x) \right) - \frac{\lambda \, x}{2! \, (1-\beta)^{2}} \, f'(x) - \frac{x}{2! \, 
(1-\beta)^{2}} \, f''(x) \right|  \\
& \hspace{10mm} \leq |\mu_{n}(x,\beta)| \, |f'(x)|+ |\nu_{n}(x,\beta)| \, |f''(x)| + |n \, R_{n}(\theta(t,x) \, \phi^{2}; x)|
\end{align*}
where
\begin{align*}
\mu_{n}(x, \beta) &= n \, R_{n}^{(\beta)}(\phi; x) - \frac{\lambda \, x}{2! \, (1-\beta)^{2}} \\
\nu_{n}(x, \beta) &= \frac{1}{2!} \, \left(n \, R_{n}^{(\beta)}(\phi^{2}; x) - \frac{x}{(1-\beta)^{2}} \right).
\end{align*}

By using \eqref{e22} it is given that
\begin{align*}
|\theta(t,x)| \leq \left( 1 + \frac{(e^{-t} - e^{-x})^{2}}{\delta^{2}} \right) \, \omega^{*}(f''; \delta)
\end{align*}
which becomes, when $|e^{-t} - e^{-x}| \leq \delta$ is taken into consideration, $|\theta(t,x)| \leq 2 \, \omega^{*}
(f''; \delta)$. If $|e^{-t} - e^{-x}| > \delta$ then $|\theta(t,x)| \leq (2/\delta^{2}) \, (e^{-t} - e^{-x})^{2} \, 
\omega^{*}(f''; \delta)$. Therefore, it can be concluded that
\begin{align*}
|\theta(t,x)| \leq 2 \, \left( 1 + \frac{(e^{-t} - e^{-x})^{2}}{\delta^{2}} \right) \, \omega^{*}(f''; \delta).
\end{align*}
The term $n \, R_{n}^{(\beta)}(\theta(t,x) \, \phi^{2}; x)$ becomes
\begin{align*}
n \, R_{n}^{(\beta)}(\theta(t,x) \, \phi^{2}; x) \leq 2 \, n \, \left( R_{n}^{(\beta)}(\phi^{2}; x) + \frac{1}{\delta^{2}} \, 
R_{n}^{(\beta)}((e^{-t} - e^{-x})^{2} \, \phi^{2}; x) \right) \, \omega^{*}(f''; \delta)
\end{align*}
which, by applying the Cauchy-Swarz inequality, becomes
\begin{align*}
n \, R_{n}^{(\beta)}(\theta(t,x) \, \phi^{2}; x) \leq 2 \, n \, \left( R_{n}^{(\beta)}(\phi^{2}; x) + \frac{1}{\delta^{2}} \, 
\zeta_{n}(x, \beta) \right) \, \omega^{*}(f''; \delta),
\end{align*}
where
\begin{align*}
\zeta_{n}(x, \beta) &= n^{2} \, \sqrt{R_{n}^{(\beta)}((e^{-x} - e^{-t})^{4}; x)} \, \sqrt{R_{n}^{(\beta)}(\phi^{4}; x)}.
\end{align*}
Now, by choosing $\delta = 1/\sqrt{n}$, the desired result is obtained. 
\end{proof}

\begin{remark}
By use of Lemma 4 it is clear that $\mu_{n}(x, \beta) \to 0$ and $\nu_{n}(x,\beta) \to 0$ as $n \to \infty$.
Using \eqref{e19} and \eqref{e23} the limit of $\zeta_{n}(x,\beta)$ becomes
\begin{align*}
\lim_{n \to \infty} \zeta_{n}(x,\beta) = \frac{3 \, x^{2} \, e^{-2 x}}{(1-\beta)^{4}}
\end{align*}
and yields
\begin{align*}
\lim_{n \to \infty} \left( 2 \, \nu_{n}(x,\beta) + \frac{x}{(1-\beta)^{2}} + \zeta_{n}(x,\beta) \right) = 
\frac{x}{(1-\beta)^{2}} + \frac{3 \, x^{2} \, e^{-2 x}}{(1-\beta)^{4}}.
\end{align*}

\end{remark}


\begin{corollary}
Let $f, f', f'' \in C^{*}[0, \infty)$ then the inequality
\begin{align*}
\lim_{n \to \infty} \, n \, \left| R_{n}^{(\beta)}(f;x) - f(x) \right| = \frac{\lambda \, x}{2! \, (1-\beta)^{2}} \, f'(x)
 + \frac{x}{(1-\beta)^{2}} \, f''(x)
\end{align*}
holds for all $x \in [0, \infty)$.
\end{corollary}


\vspace{5mm}
\section{Further Considerations}

Having established several results for the Sz\'{a}sz-Mirakyan-Jain operators further considerations can be considered. One 
such consideration could be an application of a theorem found in a recent work of Gupta and Tachev, \cite{GT}. In order to
do so the following results are required. 
\begin{lemma} \label{L6} 
Let $z_{\mu} = z(\mu/n, \beta)$, $\phi = t-x$, and $f = \text{Exp}[n \, \alpha_{n}(x) \, (z_{\mu} - 1)]$. The exponentially
weighted moments are then given by:
\begin{align}\label{e51} 
\begin{split}
R_{n}^{(\beta)}(e^{- \mu x} \, \phi^{0}; x) &= f \\
R_{n}^{(\beta)}(e^{- \mu x} \, \phi^{1}; x) &= \left[ \frac{\alpha_{n}(x) \, z_{\mu}}{1-\beta \, z_{\mu}} - x \right] \, f  \\
R_{n}^{(\beta)}(e^{- \mu x} \, \phi^{2}; x) &= \left[ \left(\frac{\alpha_{n}(x) \, z_{\mu}}{1-\beta \, z_{\mu}}
- x \right)^{2} + \frac{\alpha_{n}(x) \, z_{\mu}}{n \, (1-\beta \, z_{\mu})^{3}} \right] \, f \\
R_{n}^{(\beta)}(e^{- \mu x} \, \phi^{3}; x) &= \left[ \left(\frac{\alpha_{n}(x) \, z_{\mu}}{1-\beta \, z_{\mu}}
- x \right)^{3} + \frac{3 \, \alpha_{n}(x) \, z_{\mu}}{n \, (1-\beta \, z_{\mu})^{3}}
\, \left(\frac{\alpha_{n}(x) \, z_{\mu}}{1-\beta \, z_{\mu}} - x \right) \right. \\
& \hspace{15mm}  \left. + (1 + 2 \, \beta \, z_{\mu}) \, \frac{\alpha_{n}(x) \, z_{\mu}}{n^{2} \, (1-\beta \, z_{\mu})^{5}} 
\right] \, f \\
R_{n}^{(\beta)}(e^{- \mu t} \, \phi^{4}; x) &= \left[ \left(\frac{\alpha_{n}(x) \, z_{\mu}}{1-\beta \, z_{\mu}} -x \right)^{4} + 
\frac{6 \, 
\alpha_{n}(x) \, z_{\mu}}{n \, (1-\beta \, z_{\mu})^{3}} \, \left(\frac{\alpha_{n}(x) \, z_{\mu}}{1-\beta \, z_{\mu}} - x \right)^{2} 
\right. \\ 
& \hspace{5mm} \left. + (7 + 8 \, \beta \, z_{\mu}) \, \frac{\alpha_{n}(x) \, z_{\mu}}{n^{2} \, (1-\beta \, z_{\mu})^{5}} 
 \cdot \left(\frac{\alpha_{n}(x) \, z_{\mu}}{1-\beta \, z_{\mu}} - x \right) + (1 + 8 \beta \, z_{\mu} + 6 \beta^{2} 
\, z_{\mu}) \right. \\
& \hspace{15mm} \left. \cdot \frac{\alpha_{n}(x) \, z_{\mu}}{n^{3} \, (1-\beta \, z_{\mu})^{7}} + \frac{3 \, \alpha_{n}(x) \, 
z_{\mu}}{n^{2} \, (1-\beta \, z_{\mu})^{5}} \right] \, f
\end{split}
\end{align}
\end{lemma}

\begin{proof}
By using \eqref{e14} then
\begin{align*}
R_{n}^{(\beta)}(e^{- \mu t} \phi^{m}; x) &= n \, \alpha_{n} \, \sum_{k=0}^{\infty} \frac{1}{k!} \, (n \alpha_{n} + \beta k)^{k-1}
\, e^{-(n \alpha_{n} + \beta k)} \, e^{- \mu k/n} \, \left(\frac{k}{n} - x\right)^{m} \\
&= (-1)^{m} \, \left(\frac{d}{d \mu } + x\right)^{m}  \, e^{n \alpha_{n}(x) \, (z_{\mu} - 1)}.
\end{align*}
For the case $m =1$ it is given that
\begin{align*}
R_{n}^{(\beta)}(e^{- \mu t} \, \phi; x) &= - \left(\frac{d}{d\mu} + x\right) \, e^{n \alpha_{n}(x) \, (z_{\mu} - 1)} 
= \left[ \frac{\alpha_{n}(x) \, z_{\mu}}{1 - \beta \, z_{\mu}} - x \right] \, e^{n \alpha_{n}(x) \, (z_{\mu} - 1)}.
\end{align*}
The remainder of the moments follow.
\end{proof}

\begin{remark} \label{R6}
The ratio of $R_{n}^{(\beta)}(e^{- \mu t} \, \phi^{4}; x)$ and $R_{n}^{(\beta)}(e^{- \mu t} \, \phi^{2}; x)$ as $n \to \infty$ 
is
\begin{align}
\lim_{n \to \infty}  \frac{R_{n}^{(\beta)}(e^{- \mu t} \, \phi^{4}; x)}{R_{n}^{(\beta)}(e^{- \mu t} \, \phi^{2}; x)} = 0,
\end{align}
with order of convergence $\mathcal{O}(n^{-2})$.
\end{remark}


\begin{proof}
Consider the expansion of
\begin{align*}
\frac{\alpha_{n}(x) \, z_{\mu}}{1- \beta \, z_{\mu}} = z_{\mu} \cdot \frac{1-\beta}{1 - \beta \, z_{\mu}} \cdot \frac{\alpha_{n}(x)}
{1-\beta}
\end{align*}
by making use of the expansion used in the proof of Lemma \ref{L4}, \eqref{a3}, and by
\begin{align*}
\frac{1-\beta}{1-\beta \, z_{\mu}} = 1 - \frac{\beta \, \mu}{n(1-\beta)^{2}} + \frac{3 \, \beta^{2} \, \mu^{2}}{2! \, n^{2} (1-
\beta)^{4}} - \frac{(\beta + 14 \beta^2) \, \mu^{3}}{3! \, n^{3} (1-\beta)^{6}}  + \mathcal{O}\left(\frac{\mu^{4}}{n^{4} (1-
\beta)^{8}}\right)
\end{align*}
then 
\begin{align}
\frac{\alpha_{n}(x) \, z_{\mu}}{1 - \beta \, z_{\mu}} - x = \frac{x}{2 \, n (1-\beta)^{2}} \, \left( (\lambda - 2 \mu) + 
\frac{\sigma(\lambda, \mu)}{3! \, n (1-\beta)^{2}} + \mathcal{O}\left(\frac{1}{n^{2} (1-\beta)^{4}}\right) \right).
\end{align}
where $\sigma(\lambda, \mu) = (1-4\beta) \lambda - 6 \lambda \mu + 6(1-2\beta + 3 \beta^{2}) \mu^{2}$. By squaring this result
and taking the limit it is determined that 
\begin{align*}
\lim_{n \to \infty}  \frac{R_{n}^{(\beta)}(e^{- \mu t} \, \phi^{4}; x)}{R_{n}^{(\beta)}(e^{- \mu t} \, \phi^{2}; x)} = 
\lim_{n \to \infty} \frac{(\lambda - 2 \mu)^{2} \, x^{2}}{4 \, n^{2} (1-\beta)^{4}} \, \left( 1 + \mathcal{O}\left( \frac{1}{
n} \right) \right) \to 0.
\end{align*}
\end{proof}


With Lemma \ref{L6} and Remark \ref{R6} use could be made of Theorem 5 of Gupta and Tachev, \cite{GT}, which can be stated 
as
\begin{theorem}
Let $E$ be a subspace of $C[0,\infty)$ which contains the polynomials and suppose $L_{n} : E \to C[0,\infty)$ is a 
sequence of linear positive operators preserving linear functions. Suppose that for each constant $\mu > 0$, and fixed
$x \in [0, \infty)$, the operators $L_{n}$ satisfy
\begin{align*}
L_{n}( e^{- \mu t} \, (t-x)^{2}; x) \leq Q(\mu, x) \, R_{n}^{(\beta)}(e^{- \mu t} (t-x)^{2}; x).
\end{align*}
Additionally, if $f \in C^{2}[0, \infty) \bigcap E$ and $f^{n} \in Lip(\alpha, \mu)$, for $0 < \alpha \leq 1$, then, 
for $x \in [0, \infty)$,
\begin{align*}
& \left| L_{n}(f; x) - f(x) - \frac{f^{''}(x)}{2} \, \mu_{n,2}^{R^{(\beta)}} \right| \\
& \hspace{5mm} \leq \left[ e^{-\mu x} + \frac{Q(\mu, x)}{2} + \sqrt{\frac{Q(2 \mu ,x)}{4}} \right] \, \mu_{n,2}^{R^{(\beta)}} 
\cdot \omega_{1}\left( f^{n}, \sqrt{\frac{\mu_{n,4}^{R^{(\beta)}}}{\mu_{n,2}^{R^{(\beta)}}}}, \mu \right)
\end{align*}
where $\mu_{n,2}^{R^{(\beta)}} = R_{n}^{(\beta)}(e^{- \mu t} (t-x)^{2}; x)$.
\end{theorem}

\vspace{5mm}
\section{Appendix}
Expansion of the function $f(a e^{t})$ in powers of $t$ is is given by
\begin{align}\label{a1}
f(a e^{t}) = \sum_{k=0}^{\infty} \left[ D^{k}_{t} \, f(a e^{t}) \right]_{t=0} \, \frac{t^{k}}{k!} = f(a) + 
\sum_{k=1}^{\infty} p_{k}(a) \, \frac{t^{k}}{k!},
\end{align}
where 
\begin{align}\label{a2}
p_{n}(a) = \left[ D^{n}_{t} \, f(a e^{t}) \right]_{t=0} = \sum_{r=1}^{n} S(n, n-r) \, a^{r} \, f^{(r)}(a),
\end{align}
with $S(n,m)$ being the Stirling numbers of the second kind. Applying this expansion to the Lambert W-function
the formula $W(x e^{x}) = x$ and the $n^{th}$-derivative coefficients, Oeis A042977, \cite{Oeis1, LWM} are 
required to obtain
\begin{align}\label{a3}
- \frac{1}{\beta} \, W(- \beta \, e^{-\beta + t}) &= 1 + (1-\beta) \, \sum_{n=1}^{\infty} \frac{B_{n-1}(\beta) \, u^{n}}{n!},
\end{align}
where $(1-\beta)^{2} \, u = t$ and $B_{n}(x)$ are the Eulerian polynomials of the second kind. Let $z(t)$ be the 
left-hand side of \eqref{a3}, $- \beta \, z(t) = W(- \beta \, e^{-\beta + t})$, to obtain
\begin{align}\label{a4}
\frac{t}{(1-\beta) \, (z(t)-1)} &= 1 - \frac{u}{2!} + 2\, (1-4\beta) \, \frac{u^{2}}{4!} - 6 \beta^{2} \, \frac{u^{3}}
{4!} - (1 - 8\beta + 88\beta^{2} + 144\beta^{3}) \, \frac{u^{4}}{6!} \nonumber\\
& \hspace{5mm} - 840 \beta^{2}\,(1 + 12\beta + 8\beta^{2}) \, \frac{u^{5}}{8!} + O(u^{6}).
\end{align}
The ratio of $z(x)-1$ to $z(t)-1$ is given by
\begin{align}\label{a5}
\frac{t}{x} \, \frac{z(x)-1}{z(t)-1} &= 1 + \frac{(x-t)}{2! \, (1-\beta)^{2}} + \delta_{1} \, \frac{(x-t)}{4! \, 
(1-\beta)^{4}} +  \delta_{2} \, \frac{(x-t)}{4! \, (1-\beta)^{6}} 
+ \mathcal{O}\left(\frac{(x-t)}{8! \, (1-\beta)^{8}} \right),
\end{align}
where
\begin{align*}
\delta_{1} &= 4(1+2 \beta) x - 2(1-4\beta) \, t \\
\delta_{2} &= (1 + 8\beta + 6\beta^{2}) \, x^{2} - (1-4\beta - 6\beta^{2}) \, x t + 6 \beta^{2} \, t^{2} \\
\end{align*}


\vspace{5mm}

\end{document}